\documentclass{jams-l}
\usepackage{graphicx}
\usepackage{algorithm}

\newtheorem{theorem}{Theorem}[section]

\theoremstyle{definition}
\newtheorem{definition}[theorem]{Definition}

\theoremstyle{remark}

\numberwithin{equation}{section}

\begin{document}

\title[Function fitting]{A function fitting method}



\author{}
\address{}
\curraddr{}
\email{}
\thanks{}

\author{}
\address{}
\curraddr{}
\email{}
\thanks{}

\author{Rajesh Dachiraju}
\address{Hyderabad, India}
\curraddr{}
\email{rajesh.dachiraju@gmail.com}
\thanks{}



\dedicatory{}

\maketitle

\begin{abstract}
In this article, we describe a function fitting method that has potential applications in machine learning and also prove relevant theorems. The described function fitting method is a convex minimization problem and can be solved using a gradient descent algorithm. We also provide qualitative analysis on fitness to data of this function fitting method. The function fitting problem is also shown to be a solution of a linear, weak partial differential equation(PDE). We describe a way to fit a Sobolev function by giving a method to choose the optimal $\lambda$ parameter. We describe a closed form solution to the derived PDE, which enables the parametrization of the solution function. We describe a simple numerical solution using a gradient descent algorithm, that converges uniformly to the actual solution. As the functional of the minimization problem is a quadratic form, there also exists a numerical method using linear algebra. Lastly we give some numerical examples and also numerically demonstrate its application to a binary classification problem.
\end{abstract}




\section{Introduction}

A core issue in machine learning is fitting a function to a given dataset. In most machine learning models, training involves the need to fit a function to the training data to enable predicting function values at test data points, i.e., data points outside training samples. Examples include linear regression that requires fitting a straight line or a hyperplane to data, kernel methods, and neural networks with the mean square loss function. While the primary issue with linear regression is that the solution is always in the space of hyperplanes, kernel methods \cite{Shawe-Taylor:2004:KMP:975545} require a choice of kernel apriori and the solution depends on the chosen kernel. In the case of neural network models, the solution space is the entire space of continuous functions\cite{DBLP:journals/nn/Hornik91},\cite{DBLP:journals/mcss/Cybenko89}; however, it has two main theoretical disadvantages (a)  the solution not being unique and (b) the optimization not being a convex problem. In this article, we describe a function fitting method, which is a convex minimization problem with a unique solution. The solution space is the Sobolev space of continuous functions. Furthermore, we prove relevant theorems and describe numerical solutions that converge uniformly to the actual solution.

\section{Earlier Work on Function Fitting Methods}

\subsection{A function fitting problem}
Let $\boldsymbol{p_i}$ ($i = 1,2,3..n$) be the interior points of $(0,1)^m$ and $a_i \in \mathbb{R}$, such that $\sum\limits_{i = 1}^n a_i = 0$. $\boldsymbol{p_i}$ and $a_i$ constitute the data to which a function needs to be fit. 
\subsection{Earlier methods}

In methods such as \cite{duchon1977splines},\cite{fefferman2009fitting},\cite{fefferman2009fitting2},\cite{fefferman2009fitting3},\cite{fefferman2014fitting},\cite{fefferman2016fitting} and \cite{fefferman2016fitting2}, whose solution space is either the entire Hilbert space of smooth functions or the Sobolev space, the function is assumed to fit the data exactly, and the problem is posed as a smooth extension problem, under certain criteria like a minimum norm. In contrast, in this article, the function fitting problem is posed as a minimization problem, with a unique solution.

\section{Definitions}

Let $H^k(\Omega)$ denote the Sobolev Hilbert space on set $\Omega$,
$\mathbb{Z}^m$ denote the set containing all the $m$-tuples of integers, and $M$ denote the set of all continuous functions, defined on $\Omega = (0,1)^m$, that meet the periodic boundary conditions on the boundary $\partial \Omega$. Additionally, let $S = M\cap H^k(\Omega)$. 
\begin{definition}
\label{kgradient}
Define the $k$-gradient as 
\begin{equation}
\nabla^kf = (\frac{\partial^{k}f}{\partial x_1^{k}},\frac{\partial^{k}f}{\partial x_2^{k}},...\frac{\partial^{k}f}{\partial x_m^{k}}   ) 
\end{equation}
\end{definition}
\begin{definition}
\label{kLaplacian}
Define $k$-Laplacian  as
\begin{equation}\Delta^k f = \sum\limits_{i=1}^m \frac{\partial^{2k}f}{\partial x_i^{2k}}
\end{equation}
\end{definition}

\section{Minimization problem}

\label{Minimizer}

$\forall f \in S$, minimize the functional
\begin{equation}
\label{eq_edit_3}
C(f) = \|f\|_{T^k(\Omega)}^2 + \sum\limits_{i = 1}^n(f(\boldsymbol{p_i})-a_i)^2
\end{equation}

where \begin{equation}
\label{tkdef}
 \|f\|_{T^k(\Omega)}^2 =     \|f\|_{L^2(\Omega)}^2 + \lambda \|\nabla^k f\|_{L^2(\Omega)}^2   
\end{equation}
 and $\lambda$ is a positive real constant.

\begin{theorem}
\label{th_eq}
For this particular set $S$, the norm $\|.\|_{T^k(\Omega)}$ is equivalent to the Sobolev norm $\|.\|_{H^k(\Omega)}$.
\end{theorem}

\begin{proof}
As the norms $\|.\|_{T^k(\Omega)}$ for different $\lambda \in \mathbb{R}^+$ are equivalent, for this proof we consider only $\lambda = 1$.
Let $\boldsymbol{l} = (l_1,l_2,l_3,..l_m)\in \mathbb{Z}^m$ and $\alpha$ a multi-index. Let $u_{\boldsymbol{l}}$ be the Fourier series coefficients of $u\in S$, we have

\begin{equation}
    ||u||_{H^k(\Omega)}^2 = ||u||_{L^2(\Omega)}^2 + \sum_{|\alpha| =  k} ||D^{\alpha} u||_{L^2(\Omega)}^2.
\end{equation}
By Plancherel's theorem
\begin{equation}
\sum_{|\alpha| =  k} ||D^{\alpha} u||_{L^2(\Omega)}^2 =   \sum_{|\alpha| =  k} \sum_{\boldsymbol{l} \in \mathbb{Z}^k} ((2\pi)^k \boldsymbol{l^{\alpha}})^{2} |\hat{u}_{\boldsymbol{l}}|^2
=
 \sum_{\boldsymbol{l} \in \mathbb{Z}^k}( |\hat{u}_{\boldsymbol{l}}|^2 \sum_{|\alpha| =  k} ((2\pi)^k \boldsymbol{l^{\alpha}})^{2} )   
\end{equation}
By arithmetic mean-geometric mean inequality, it can be shown that 
\begin{equation}
    \sum_{|\alpha| =  k} ((2\pi)^k \boldsymbol{l^{\alpha}})^{2} \le C_k \sum_{i =  1}^m (2\pi l_i)^{2k}
\end{equation}
with $C_k$ depending only on $k$.
So 
\begin{equation}
    \sum_{|\alpha| =  k} ||D^{\alpha} u||_{L^2(\Omega)}^2 \le C_k \sum_{\boldsymbol{l} \in \mathbb{Z}^k}( |\hat{u}_{\boldsymbol{l}}|^2 \sum_{i = 1}^m (2\pi l_i)^{2k} ) = C_k \sum_{i = 1}^m( \sum_{\boldsymbol{l} \in \mathbb{Z}^k}(2\pi l_i)^{2k}|\hat{u}_{\boldsymbol{l}}|^2 ) 
\end{equation}
Using equation \ref{tkdef} and applying Plancherel's theorem in reverse
\begin{equation}
    \|u\|_{L^2(\Omega)}^2 + \sum_{i = 1}^m( \sum_{\boldsymbol{l} \in \mathbb{Z}^k}(2\pi l_i)^{2k}\hat{u}_{\boldsymbol{l}}^2 ) = \|u\|_{T^k(\Omega)}
\end{equation}
Therefore
\begin{equation}
\|u\|_{H^k(\Omega)} \le D_k \|u\|_{T^k(\Omega)}    
\end{equation}
where $D_k$ a constant depending only on $k$.
We can easily observe that $\|u\|_{H^k(\Omega)} \ge \|u\|_{T^k(\Omega)} $. Hence the norms are equivalent.
\end{proof}

\begin{theorem}
\label{th_uconv}
Given that $k>\frac{m}{2}$, If $u \in H^k(\Omega)$, then
\begin{equation}
u\in L^{\infty}(\Omega)
\end{equation}
and 
\begin{equation}
\|u\|_{L^{\infty}(\Omega)} \le K\|u\|_{H^k(\Omega)}
\end{equation}
with $K$  depending only on $k$ and $m$
\end{theorem}

\begin{proof}
Let us express $u$ in terms of its Fourier series coefficients $\hat u_{\boldsymbol{l}},\boldsymbol{l} \in \mathbb{Z}^m$, via the Fourier series expansion and then the trick is to multiply by 1 in disguise, with $\langle \boldsymbol{l}\rangle := \sqrt{1+|\boldsymbol{l}|^2}$  
\begin{equation}
u(\boldsymbol{x}) = \sum\limits_{\boldsymbol{l}\in \mathbb{Z}^m} \hat u_{\boldsymbol{l}}e^{2\pi i \boldsymbol{l}\cdot \boldsymbol{x}}  = \sum\limits_{\boldsymbol{l}\in \mathbb{Z}^m} \hat u_{\boldsymbol{l}} \langle \boldsymbol{l}\rangle^k \langle \boldsymbol{l}\rangle^{-k} e^{2\pi i \boldsymbol{l}\cdot \boldsymbol{x}} 
\end{equation}
by Hölder's inequality,
\begin{equation}
 |u(\boldsymbol{x})| \le \sum\limits_{\boldsymbol{l}\in \mathbb{Z}^m} \left |\hat u_{\boldsymbol{l}}\langle \boldsymbol{l}\rangle^k \right| \langle \boldsymbol{l}\rangle^{-k} \le \sqrt{\sum\limits_{\boldsymbol{l}\in \mathbb{Z}^m}|\hat u_{\boldsymbol{l}}\langle \boldsymbol{l}\rangle^{k}|^2 \sum\limits_{\boldsymbol{l}\in \mathbb{Z}^m}|\langle \boldsymbol{l}\rangle^{-k}|^2}   
\end{equation} 

By Plancherel's Theorem, $\sqrt{\sum\limits_{\boldsymbol{l}\in \mathbb{Z}^m}|\hat u_{\boldsymbol{l}}\langle \boldsymbol{l}\rangle^{k}|^2} = \|u\|_{H^{k}}$ and 
$K = \sqrt{\sum\limits_{\boldsymbol{l}\in \mathbb{Z}^m}|\langle \boldsymbol{l}\rangle^{-k}|}$ is a constant depending only on  $k,n$, which is finite as $k > m/2$. This completes the proof.
\end{proof}

\begin{theorem}
\label{completeness}
Given that $k>\frac{m}{2}$, any sequence in $S$, that converges in the norm $\|.\|_{T^k}$, also converges uniformly to a limit function in $S$.
\end{theorem}

\begin{proof}
Let $\{f_n\}\to f$ under the norm $\|.\|_{T^k}$, then $\|f_n-f\|_{T^k}\to 0$, so $\|f_n-f\|_{H^k} \to 0$, (as $\|.\|_{T^k}$ is equivalent to $\|.\|_{H^k}$ due to Theorem \ref{th_eq}) and hence due to Theorem \ref{th_uconv}, $\|f_n-f\|_{L^{\infty}(\Omega)} \to 0$. So, as this sequence of continuous functions with periodic boundary conditions converges uniformly, the limit function $f$ is also a continuous function with periodic boundary conditions and so $f\in M $. It is evident that $f \in H^k(\Omega)$, so $f\in S$. 
\end{proof}

\begin{theorem}
Given that $k>\frac{m}{2}$, the minimizer of the functional $C(f)$ over the set $S$ exists and is unique.
\end{theorem}

\begin{proof}
Let $\delta$ be the infimum of $C(f)$ over the set $S$. So there exists a sequence $\{f_n\}, f_n \in S$ such that $C(f_n) \to \delta$. Since both terms of $C(f)$ are positive, due to first term, $\{f_n\}$ is Cauchy under the norm $\|.\|_{T^k}$. Due to theorem \ref{completeness}, $S$ is a closed linear subspace of the Hilbert space $H^k$, and with the inner product induced by restriction, is also a Hilbert space in its own right. Hence the sequence $\{f_n\}$ converges to a limit function $g \in S$ under the norm $\|.\|_{T^k}$, which also means
\begin{equation}
\label{eq_edit_1}
\|f_n\|_{T^k} \to \|g\|_{T^k}
\end{equation}
Again due to theorem \ref{completeness}, $\{f_n\} \to g$ pointwise. So \begin{equation}
\label{eq_edit_2}
f_n(\boldsymbol{p_i}) \to g(\boldsymbol{p_i}), i = 1,2,..N
\end{equation}
Using Equations \ref{eq_edit_3},\ref{eq_edit_1},\ref{eq_edit_2}, we can say that $C(f_n)\to C(g)$, and therefore $C(g) = \delta$. Hence as $g \in S$, the infimum of $C(f)$ over set $S$ is attained in $S$. Uniqueness follows from the convexity of $C(f)$.
\end{proof}
This proves the existence and uniqueness of the solution to the minimization problem.

\section{Euler-Lagrange Equation}

We now derive the Euler-Lagrange equation of the minimization problem posed in earlier section, and show that it is a linear weak PDE with some global terms.

Minimize in $S$, \begin{equation} C(f) = \sum\limits_{i = 1}^n(f(\boldsymbol{p_i})-a_i)^2 +   \|f\|_{L^2(\Omega)}^2 + \lambda \|\nabla^k f\|_{L^2(\Omega)}^2\end{equation}

We will derive the Euler Lagrange equation for the above problem, step by step for each term separately.
For any $\phi \in C^{\infty}(\Omega)\cap M$

\begin{equation}\frac{d}{ds}|_{s=0} \|f+s\phi\|_{L^2(\Omega)}^2 = \frac{d}{ds}|_{s=0} \int_{\Omega} |f+s\phi|^{2} \stackrel{*}{=} 2\int_{\Omega} \phi f\end{equation} where $*$ can be justified by using the dominated convergence theorem.

\begin{equation}\frac{d}{ds}|_{s=0} \lambda\| \nabla^k(f+s\phi)\|_{L^2(\Omega)} = \frac{d}{ds}|_{s=0} \int_\Omega \lambda|\nabla^k f+s \nabla^k \phi|^2 = 2\lambda\int_\Omega \nabla^k \phi \cdot \nabla^k f \end{equation}

\begin{equation}\frac{d}{ds}|_{s=0} \sum\limits_{i = 1}^n|f(\boldsymbol{p_i})+s\phi(\boldsymbol{p_i})-a_i|^2 =  2 \sum\limits_{i = 1}^n (f(\boldsymbol{p_i})-a_i)\phi(\boldsymbol{p_i})
\end{equation}

and putting all terms together, we get the following PDE as the Euler-Lagrange equation for the minimization problem.

\begin{equation}
\label{tkdef2}
\lambda\int_{\Omega}\nabla^k\phi(\boldsymbol{x})\cdot \nabla^k f(\boldsymbol{x}) \mathop{}\!\mathrm{d}^m\boldsymbol{x}+  \int_{\Omega} \phi(\boldsymbol{x}) f(\boldsymbol{x})\mathop{}\!\mathrm{d}^m\boldsymbol{x}   + \sum\limits_{i = 1}^{n}(f(\boldsymbol{p_i})-a_i)\phi(\boldsymbol{p_i}) = 0 \forall \phi \in C^{\infty}(\Omega) \cap M
\end{equation}


The equation is not a PDE in strict sense, due to the appearance of global terms in it, like $f\boldsymbol{(p_i}) = \int_{\Omega} f(\boldsymbol{x})\delta(\boldsymbol{x-p_i})$.

\section{Fitness to Data Analysis}
\label{FDA}

In this section we describe how to control the fitness of the function on data points, by controling the value of $\lambda$. Let $f_{\lambda}$ be the solution of the Euler Lagrange equation, i.e the minimizer of the minimization problem in section \ref{Minimizer}.

\begin{theorem}
\label{du_th}
If $f_{\lambda}$ is the minimizer of $C(f)$, then \begin{equation}
  \lim\limits_{\lambda\to0}f_{\lambda}(\boldsymbol{p_i}) = a_i, i = 1,2...n
\end{equation}
and 

\begin{equation}
   \lim\limits_{\lambda\to 0}f_{\lambda}(\boldsymbol{x}) = 0 \text{ almost everywhere}
\end{equation}

\end{theorem}

\begin{proof}

Let $B^r_i$ be balls of radius $r$ around points $\boldsymbol{p_i}$, and let $B^r = \bigcup\limits_{i = 1}^nB^r_i$

Given any $g \in S$, fixing $g$, we can see that  
\begin{equation}
\label{tlo}
   \lim\limits_{\lambda\to 0}C(g) = \sum\limits_{i = 1}^n(g(\boldsymbol{p_i})-a_i)^2 +   \|g\|_{L^2(\Omega)}^2 
\end{equation}

Consider the function \begin{equation} \theta_r = \sum\limits_{i = 1}^n\phi_i \end{equation} where $\phi_i$ is a bump function with support on the ball $B^r_i$ and also $\phi_i(\boldsymbol{p_i}) = a_i$
Therefore

\begin{equation}
\label{tl}
   \lim\limits_{\lambda\to 0}C(\theta_r) = \sum\limits_{i = 1}^n(\theta_r(\boldsymbol{p_i})-a_i)^2 +   \|\theta_r\|_{L^2(\Omega)}^2 
\end{equation}

For any given $\lambda$, let $f_{\lambda}$ denote the minimizer of the functional $C(f)$. By definition of $f_{\lambda}$, $C(f_{\lambda})\le C(\theta_r)$,

\begin{equation}
\label{tm}
   \lim\limits_{\lambda\to 0}C(f_{\lambda}) \le \lim\limits_{\lambda\to 0}C(\theta_r)
   \forall r
\end{equation}
therefore 
\begin{equation}
\label{tp}
   \lim\limits_{\lambda\to 0}C(f_{\lambda}) \le \lim\limits_{r\to 0}\lim\limits_{\lambda\to 0}C(\theta_r)
\end{equation}

Now using Equation \ref{tl}, its easy to see that \begin{equation}
\label{tt2}
\lim\limits_{\lambda\to 0}\lim\limits_{r\to 0}C(\theta_r)  = \lim\limits_{r\to 0}\{ \sum\limits_{i = 1}^n(\theta_r(\boldsymbol{p_i})-a_i)^2 +   \|\theta_r\|_{L^2(\Omega)}^2 \} = 0 
\end{equation}
So using Equations \ref{tp} and \ref{tt2}
\begin{equation}
    \lim\limits_{\lambda\to 0}C(f_{\lambda}) = 0
\end{equation}

Hence each term of $C(f_{\lambda})$ should go to $0$ as $\lambda \to 0$, which gives the following results

\begin{equation}
\label{de1}    \lim\limits_{\lambda\to 0} \sum\limits_{i = 1}^n(f_{\lambda}(\boldsymbol{p_i})-a_i)^2 = 0
\end{equation}

\begin{equation}
    \lim\limits_{\lambda\to 0} \| f_{\lambda}\|_{L^2(\Omega)} = 0
\end{equation}

and 

\begin{equation}
\label{d2}
    \lim\limits_{\lambda\to 0}\lambda \|\nabla^k f_{\lambda}\|_{L^2(\Omega)} = 0
\end{equation}

As all terms in Equation \ref{de1} are positive,
\begin{equation}
\label{des1}
\lim\limits_{\lambda\to 0} f_{\lambda}(\boldsymbol{p_i}) = a_i
\end{equation} 
which proves the first statement of the theorem.

To prove the second statement of the theorem, at first from equation \ref{de1} and the fact that $f_{\lambda}$ satisfies the Euler-Lagrange Equation \ref{tkdef2},

\begin{equation}
\label{tpp}
    \lim\limits_{\lambda\to 0}\lambda\int_{\Omega}\nabla^k\phi(\boldsymbol{x})\cdot \nabla^k f_{\lambda}(\boldsymbol{x}) \mathop{}\!\mathrm{d}^m\boldsymbol{x} =  -\lim\limits_{\lambda\to 0} \int_{\Omega} \phi(\boldsymbol{x}) f_{\lambda}(\boldsymbol{x})\mathop{}\!\mathrm{d}^m\boldsymbol{x}   \forall \phi \in C^{\infty}(\Omega) \cap M
\end{equation}

By Cauchy-Schwartz inequality, there exists a positive real $L$ such that \begin{equation}
    \label{CS}
    \lambda\int_{\Omega}\nabla^k\phi(\boldsymbol{x})\cdot \nabla^k f_{\lambda}(\boldsymbol{x}) \mathop{}\!\mathrm{d}^m\boldsymbol{x} \le L \lambda \|\nabla^k f_{\lambda}\|_{L^2(\Omega)}\|\nabla^k \phi\|_{L^2(\Omega)}
\end{equation} 

Now using Equations \ref{de1} and \ref{d2} \begin{equation}
\label{tq}
\lim\limits_{\lambda\to 0}\lambda\int_{\Omega}\nabla^k\phi(\boldsymbol{x})\cdot \nabla^k f_{\lambda}(\boldsymbol{x}) \mathop{}\!\mathrm{d}^m\boldsymbol{x} = 0
\end{equation}

Hence from Equations \ref{tpp} and \ref{tq}
\begin{equation}
    \lim\limits_{\lambda\to 0} \int_{\Omega} \phi(\boldsymbol{x}) f_{\lambda}(\boldsymbol{x})\mathop{}\!\mathrm{d}^m\boldsymbol{x} = 0   \forall \phi \in C^{\infty}(\Omega) \cap M
\end{equation}
This means that $ \lim\limits_{\lambda\to 0}f_{\lambda}(x) = 0  \text{ almost everywhere}$, which completes the proof for the later statement of the theorem.

\end{proof}

\subsection{Trade-off}
To increase the fitness on the data points, the value of $\lambda$ needs to be decreased. However, due to the second statement of theorem \ref{du_th}, this decrease will cause the undesirable effect of the function going to zero, almost everywhere other than the data points. To reduce this effect of function concentration on the data points, $\lambda$ needs to be increased. However, based on the first statement of theorem \ref{du_th}, this increase could cause a loss of fitness on the data points. Therefore, a trade-off between the fitness on the data points and the spread of the function over $\Omega$ is essential. This trade-off can be achieved by appropriately setting $\lambda$.

\section{Method for Fitting a Sobolev function}

\label{Method}

Consider the family of functionals over the parameter $\lambda \in (0.\infty)$
\begin{equation}
C_{\lambda}(f) = \sum\limits_{i = 1}^n(f(\boldsymbol{p_i})-a_i)^2 + \|f\|_{L^2(\Omega)}^2 + \lambda \|\nabla^k f\|_{L^2(\Omega)}^2
\end{equation}
where  $k \in \mathbb{N}, k > \frac{m}{2}$

Let $f_{\lambda}$ denote the unique minimizer of $C_{\lambda}(f)$ over the set $S$. Let a maximum of $\|f_{\lambda}\|_{L^2(\Omega)}$ over $\lambda \in (0,\infty)$ be achived at $\lambda = \lambda_0$. The function $h = f_{\lambda_0}$ is the desired function that fits the data(not necessarily a perfect fit to the data points).

\subsection{Existence and Uniqueness}

It has been shown in Section \ref{Minimizer} that $f_{\lambda}$ is unique for any given $\lambda \in (0,\infty)$. Relevant theorems are also included in the appendix section of this paper. To show the existence of $h = f{_{\lambda_0}}$, we can prove the following facts.

It has been deduced in Section \ref{FDA} that
$$\lim_{\lambda \to 0} \|f_{\lambda}\|_{L^2(\Omega)} = 0 $$

and using Poincare's inequality\cite{evans10}, one can deduce that

$$\lim_{\lambda \to \infty} \|f_{\lambda}\|_{L^2(\Omega)} = 0 $$

To prove uniqueness of $h$, it remains to be shown that $\lambda_0$ is unique.

\section{Parametrization of the Minimizer using Closed form Solution to the PDE}
\label{DCS}

In this section we describe a closed form solution to the derived PDE, which enables the parametrization of the solution function.

Let \begin{equation}g_{\lambda}(\boldsymbol{x}) = \sum\limits_{{\pmb{\eta}\in\mathbb{Z}^m}}\frac{1}{1+\lambda\sum\limits_{i=1}^{m}\eta_i^{2k}} \cos({2\pi \pmb{\eta}\cdot\pmb{x}})\end{equation}

Assuming $k >\frac{m}{2}$, using Bochner's theorem we can see that the function $g_{\lambda}$ is positive definite. The solution to the PDE (the Euler-Lagrange equation for the minimization problem)  is then given as

\begin{equation}f_{\lambda}(\pmb{x}) = \sum\limits_{i=1}^{n}c_ig_{\lambda}(\pmb{x-p_i}) \end{equation} where $c_i \in \mathbb{R}$.
Let $\pmb{c} = [c_1,c_2,...c_n]^T$. We can determine $\pmb{c}$ by substituting the above expression for $f_{\lambda}(\pmb{x})$ in the PDE equation and is given as \begin{equation}\pmb{c} = (G_{\lambda}+I)^{-1}L\end{equation} where the matrix $G_{\lambda}$ is given as $G_{\lambda} = [\gamma_{ij}]_{n\times n},\gamma_{ij} = g_{\lambda}(\pmb{p_i}-\pmb{p_j})$ and $L = [a_1,a_2,....a_n]^T$. The matrix $G_{\lambda}$ is positive semidefinite as $g_{\lambda}$ is a positive definite function. Hence the matrix $G_{\lambda}+I$ is positive definite and is invertible.

One can see that it has close similarities with regularized kernel based regression \cite{poggio1990networks},\cite{belkin2006manifold}. The difference is, here the parameter lambda is not the usual regularization parameter, and by changing lambda we are not changing the regularization of the kernel, but varying the lambda, we are navigating between different kernels, with regularization parameter being at constant and equal to 1.

\section{Numerical Solution to the Minimization Problem}

While we can numerically solve the PDE using the closed form solution given in Section \ref{DCS}, alternatively, instead of solving the PDE, we can directly solve the minimization problem. As the minimization problem is convex, it can be solved using a gradient descent algorithm. As $S$ is as the Hilbert space, due to the Plancheral theorem, optimization can be directly applied on the Fourier series coefficients. First, we discretize the domain into uniformly-spaced samples sampled at a frequency $\boldsymbol{\omega}$Hz. We also discretize the data, so that each data points is mapped to one of into uniformly spaced discrete samples of the domain. If multiple data points fall into the same discrete sample, the average of the values of the data points is assumed. We then compute Nyquist sampled version $f_{\boldsymbol{\omega}}$ of solution $f$, by expressing $C(f_{\boldsymbol{\omega}})$ in terms of the Discrete Fourier Transform (DFT) coefficients of $f_{\boldsymbol{\omega}}$. For the computation, Plancheral theorem is used and $C(f_{\boldsymbol{\omega}})$ is minimized by applying a gradient descent algorithm to the DFT coefficients. By choosing sufficiently high $\omega$, we can compute the sampled version of $f$ to desired accuracy. Due to uniform convergence of the Fourier series, the numerical solution converges uniformly to the actual solution $f$ as $\boldsymbol{\omega} \to \infty$.

\section{Numerical Solution using Linear Algebra}

There exists an alternate solution using linear algebra. $C(f_{\omega})$ can be expressed in terms of the Discrete Fourier Transform (DFT) coefficients of $f_{\omega}$ using Plancherel theorem. Taking the derivatives of the functional with respect to each of the DFT coefficients and setting the gradient to zero, we obtain a set of linear equations equal in number to the DFT coefficients. These simultaneous equations can be expressed in a matrix form. The solution can be expressed as:
\begin{equation}
A = (X^TX + I + \lambda M)^{-1}X^TL,
\end{equation}
where $A$ is a column vector of the DFT coefficients; $X$ is the transformation matrix corresponding to the $m$-dimensional DFT, with only columns pertaining to data points $p_i$ being retained; $I$ an identity matrix; $M$ is a diagonal matrix with entries that act as a Fourier multiplier corresponding to the magnitude of the $k$-gradient; and $L$ is a column vector containing data values $a_i$.

%


\section{Numerical Algorithm for choosing the optimal $\lambda$}

An Iterative algorithm which is a modified steepest descent, is run on $\hat{f}$, the Fourier Coefficients (DFT) of the function $f$. This algorithm simultaeously finds $\lambda_0$ (for which $\|f_{\lambda}\|_{L^2(\Omega)}$ is maximum) and also DFT coefficients $\hat{f}_{\lambda_o}$, from which $f_{\lambda_o}$ can be obtained using inverse DFT.
\begin{algorithm}

\begin{enumerate}

\item Initialize $\hat{f}$.
 
\item Assuming some $\lambda$ and assuming gradient of $C_{\lambda}(\hat{f})$ wrt $\hat{f}$ be $\nabla_{\hat{f}} C_{\lambda}(\hat{f})$, and if we were to update $\hat{f}$
with this gradient as in we do in steepest descent, it would be
$\hat{f}^u_\lambda = \hat{f} - \delta \nabla_{\hat{f}} C_{\lambda}(\hat{f})$, where $\delta$ is
a constant learning rate. Now set
$\frac{\partial\|\hat{f}^u_\lambda\|}{\partial \lambda} = 0$ and solve for
$\lambda$. Let the root be $\lambda_0$.
\item Update $\hat{f} = \hat{f}^u_{\lambda_0}$. (update $f$ as in steepest descent, but using $\lambda$ value as $\lambda_0$ which was computed in step
2.)
\item check some convergence criterion and if not met, go to step 2.

\end{enumerate}
\end{algorithm}
 
\section{Numerical Examples}

A 1 dimensional example, with green vertical lines indicating data points and blue plot showing the fitted function $h(x)$
\begin{center}
\includegraphics[scale = 0.22]{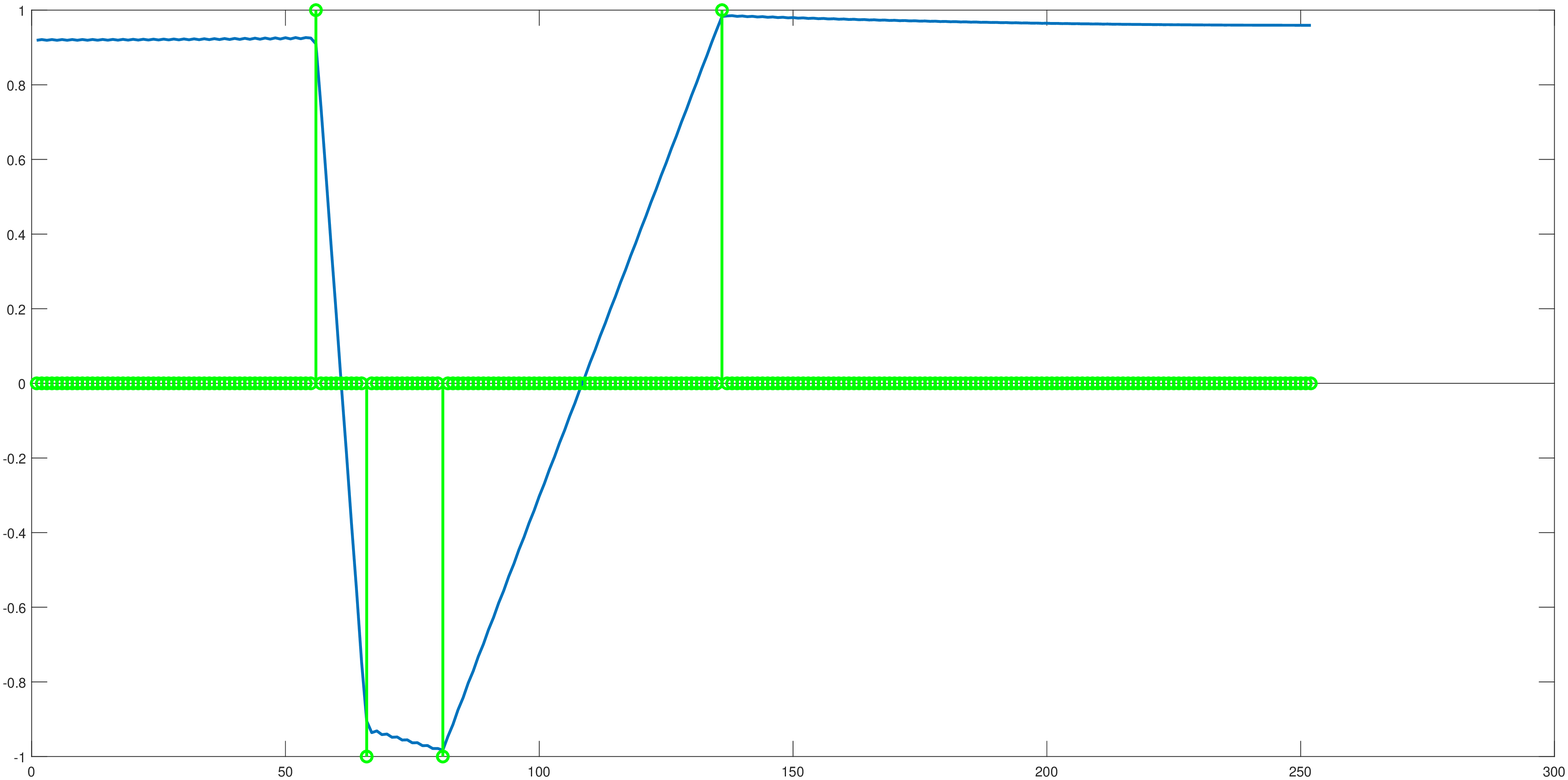}
\end{center}

We experimented with IRIS Dataset\cite{fisher1936138}, obtained from UCI repository\cite{Lichman:2013}, for a binary classification task (one class Vs all others).Given below are he specific details of the experiment.

\begin{enumerate}
\item 6 dimensional features, reduced to effectively 3 dimensions after feature space quantization.
\item First the features are normalized (zero meaned and scaled to lie between -1 and 1)
\item PCA is performed and the co-ordinate system with axes along the PCA axes is choosen. Feature space is quantized with number of quantization steps (or sampling frequency) along each direction is chosen to be proportional to the standard deviation (eigen value) along that direction).
\item After quantization we obtain a 61x23x9 multidimensional array to hold the feature space. Every example (feature vector) is mapped to one of the points in this array.
\item In training we estimate the DFT (61x23x9 point) coefficients of the desired function.
\item Positive class examples are assigned +1 as the data value and negative class examples are assigned 0. After that the data values are normalized to zero mean (as required by the theorem). Binary decision is  made in a way that positive valued examples are treated as positive class and negative valued examples are treated as negative class.
\end{enumerate}
The experiment achieved a perfect accuracy of $100$. A 1-dimenional plot of the solution function along an line parallel to longest axis, that passes through 3 data points is shown in the below picture. (green vertical lines indicating data points present on that line, and blue plot showing the fitted function $h(x)$) .

\begin{center}
\includegraphics[scale = 0.40]{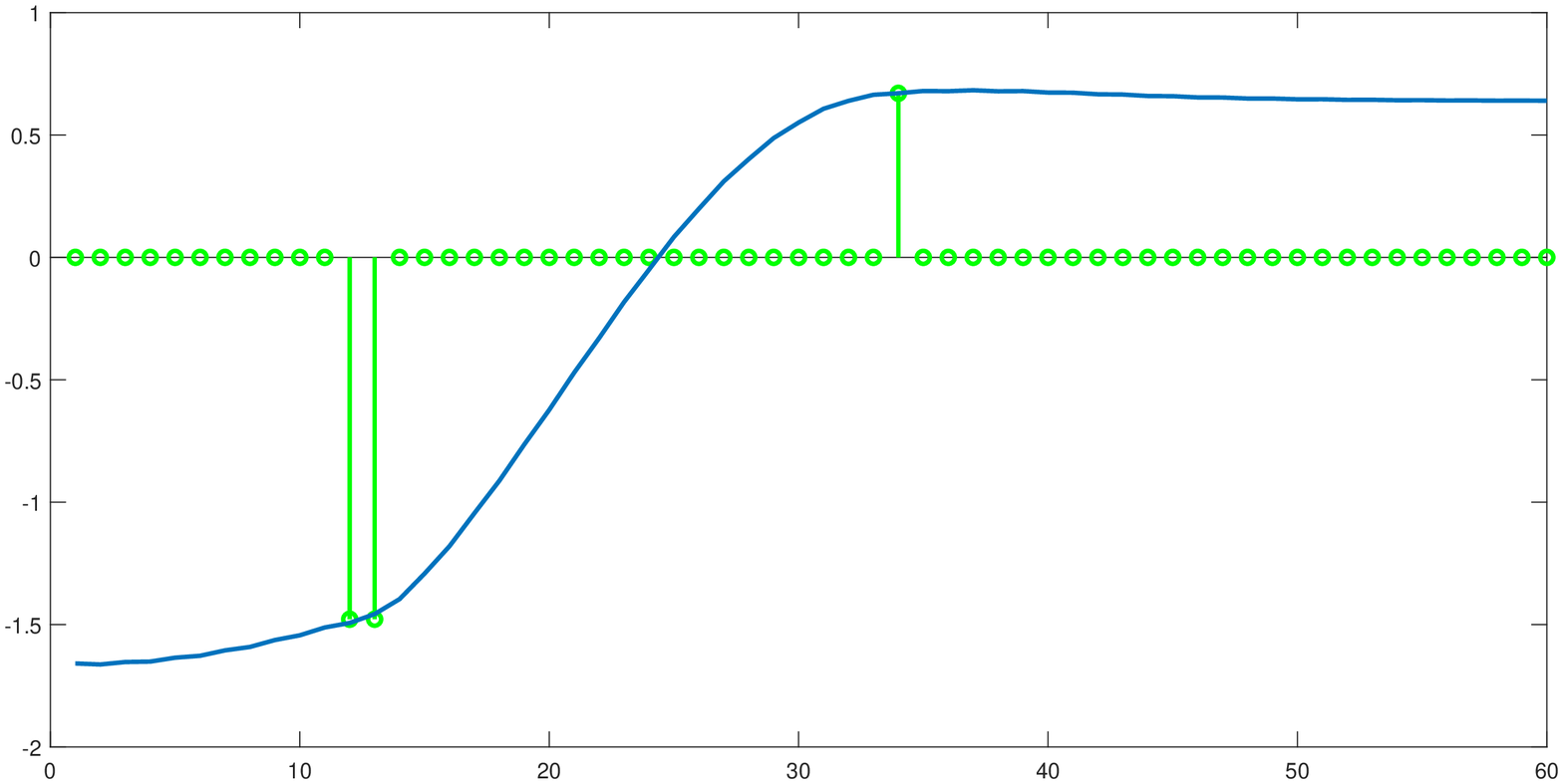}
\end{center}

\section{Applications}

This method can be used in machine learning because, in almost all problems in the field, there is a direct or indirect need for fitting a function to the data. As the functions in machine learning need not be periodic, the method can be used for non-periodic functions by even symmetric extension of both the domain and the data. 





\bibliographystyle{amsplain}
\bibliography{refs}

\end{document}